\documentclass[a4paper,10pt]{amsart}

\usepackage[T1]{fontenc}

\usepackage{amsrefs}

\usepackage{microtype}

\newtheorem{proposition}{Proposition}
\newtheorem{theorem}[proposition]{Theorem}
\newtheorem{lemma}[proposition]{Lemma}

\title[Existence of eigenvalues and eigenvectors]{Proving the existence of eigenvalues and eigenvectors by Weierstrass's theorem}
\author{Jean Van Schaftingen}
\address{Universit\'e catholique de Louvain\\
Institut de Recherche en Math\'ematique et Physique (IRMP)\\
Chemin du Cyclotron 2 bte L7.01.01\\
1348 Louvain-la-Neuve\\
Belgium}
\email{Jean.VanSchaftingen@uclouvain.be}
\date{\today}
\subjclass[2010]{15A18}
\keywords{Spectral theory; eigenvalue; eigenvector; Weierstrass's theorem; linear operators on complex vector spaces} 
\newcommand{\C}{\mathbb{C}}
\newcommand{\R}{\mathbb{R}}
\newcommand{\N}{\mathbb{N}}

\newcommand{\norm}[1]{\lVert #1 \rVert}
\newcommand{\bignorm}[1]{\bigl\lVert #1 \bigr\rVert}
\newcommand{\abs}[1]{\lvert #1 \rvert}
\newcommand{\st}{\, : \, }

\begin{document}

\begin{abstract}
I propose a proof of the existence of the existence of eigenvectors and eigenvalues in the spirit of Argand's proof of the fundamental theorem of algebra. The proof only relies on Weierstrass's theorem, the definition of the inverse of a linear operator and algebraic identities.
\end{abstract}

\maketitle

A fundamental result in the spectral theory of matrices or operators on finite-dimensional spaces is the existence of eigenvectors and eigenvalues.

\begin{theorem}
\label{theoremMain}
Let \(V\) be a finite-dimensional complex vector space and let \(T : V \to V\) be a linear operator. If \(V \ne \{0\}\), there exists \(\Bar{v} \in V \setminus \{0\}\) and \(\Bar{\lambda} \in \C\) such that 
\[
  T(\Bar{v})= \Bar{\lambda} \Bar{v}\:.
\]
\end{theorem}

The vector \(\Bar{v}\) is then an eigenvector with associated eigenvalue \(\Bar{\lambda}\). This theorem is the first step in the study of spectral decomposition and invariant subspaces of \(T\).

The traditional proof that I was taught as an undergraduate student relies on the theory of determinants.
One first remarks that \(\lambda \in \C\) is an eigenvalue if and only if the characteristic polynomial vanishes: \( \det (T - \lambda I) = 0\). Note then that this is a polynomial equation; by the fundamental theorem of algebra, there exists \(\lambda \in \C\) such that \( \det (T - \lambda I) = 0\).

You can also prove the existence of eigenvalues by relying only on the consequence of the intermediate value theorem that every real odd-degree polynomials has a real root. By the argument that I have just sketched linear operators on odd-dimensional real vector spaces  have a real eigenvalue. You can then deduce algebraically the existence of eigenvalues for linear operators on any complex vector space \cite{Derksen}.

If you want you can also avoid determinants. For \(v \in V \setminus\{0\}\) and \(d = \dim V\), the vectors \(v, T(v), T^2(v), \dotsc, T^n(v)\) are to be linearly independent, that is \((a_n T^n + \dotsb + a_0 I)(v) = 0\) for some \(a_0, \dotsc, a_k \in \C\). By the fundamental theorem of algebra this can be rewritten as
\( (T-\lambda_1) \circ (T-\lambda_2) \circ \dotsb  \circ (T-\lambda_r)(v)= 0\) for some \(r \in \{1,\dotsc, n\}\) and \(\lambda_1, \lambda_2,\dotsc, \lambda_r \in \C\).  and \( (T-\lambda_j) \) is not invertible for some \(i \in \{1, \dotsc, r\}\) \citelist{\cite{AxlerBook}*{theorem 5.10}\cite{Axler}}.

\bigskip

All the approaches that I have presented rely on the existence of roots for polynomials. 
I would like to point out that eigenvector and eigenvalue are defined without reference to polynomials and that they are computed numerically most of the time without computing roots of polynomials. This suggests to search for polynomial-free proofs of the existence of eigenvectors and eigenvalues.

Indeed, there is another family of proofs coming from the spectral theory of normed algebras \citelist{\cite{Mazur}\cite{Mazet}\cite{Gelfand}}. The idea is to assume that the function \(\lambda \in \C \mapsto (T - \lambda I)^{-1}\) is well-defined and to derive a contradiction by Liouville's theorem, the maximum modulus principle, or some direct argument \citelist{\cite{Kametani}\cite{Tornheim}\cite{Singh}}.

I am proposing here a proof of theorem~\ref{theoremMain} that does not require any previous knowledge of polynomials and algebra of operators. 
The tools that are needed are quite elementary: continuous functions of several variables, definition of the inverse of an invertible linear map and summation formulas for geometric progressions.

I am following Argand's strategy to prove the fundamental theorem of algebra \cite{Numbers}*{\S 4.2}.  Given a polynomial \(P\), he minimizes \(z \in \C \mapsto \abs{P(z)}\) and then shows that the minimum value of that function should be \(0\). Here I minimize instead the function \( (v, \lambda) \in (V \setminus \{0\}) \times \C \to \frac{\norm{T(v)-\lambda v}}{\norm{v}}\). The problem of the existence of an eigenvector and an eigenvalue, consists in showing that the minimum is \(0\). 

\bigskip

Let us begin the proof consists by establishing the existence of a minimal pair \( (\Bar{v}, \Bar{\lambda})\). For the sake of the proof we endow \(V\) with an arbitrary norm.

\begin{proposition}
\label{propositionExistenceMinimum}
Let \(V\) be a finite-dimensional normed complex vector space. If \(T : V \to V\) is a linear operator and \(V \ne \{0\}\), then there exists \(\Bar{v} \in V \setminus \{0\}\) and \(\Bar{\lambda} \in \C\) such that for every \(v \in V \setminus \{0\}\) and \(\lambda \in \C\)
\[
  \frac{\norm{T(v)-\lambda v}}{\norm{v}}\ge \frac{\norm{T(\Bar{v})-\Bar{\lambda} \Bar{v}}}{\norm{v}}\:.
\]
\end{proposition}

Since the domain of the function that we are minimizing is not compact, we need a coercivity estimate on \(\lambda\) that will be provided by the boundedness of linear operators on finite-dimensional spaces.

\begin{lemma}
\label{lemmaInequality}
Let \(V\) be a finite-dimensional normed complex vector space. If \(T : V \to V\) is a linear operator, then there exists \(C \in [0, \infty) \) such that for every \(v \in V\)
\[
 \norm{T(v)}\le C \norm{v}
\]
and for every \(v \in V\) and \(\lambda \in \C\),
\[
 \norm{T(v)-\lambda v} \ge (\abs{\lambda}-C)\norm{v}\:.
\]
\end{lemma}
\begin{proof}
The first assertion is classical for linear operators on finite-dimensional normed vector-spaces.
For the second assertion, by the triangle inequality,
\[
 \abs{\lambda}\norm{v} = \norm{\lambda v} \le \norm{T(v)-\lambda v} + \norm{T(v)}
\le \norm{T(v)-\lambda v} + C \norm{v}\:.\qedhere
\]
\end{proof}

We can now prove the existence of a minimal pair.

\begin{proof}[Proof of proposition~\ref{propositionExistenceMinimum}]
Define the function \(f : V \times \C \to \R\) for \(v \in V\) and \(\lambda \in \C\) by
\[
  f(v, \lambda)= \norm{T(v)- \lambda v}\:.
\]
Choose \(v_0 \in V\) with \(\norm{v_0}=1\). One has for every \(v \in V\) with \(\norm{v}=1\), by lemma~\ref{lemmaInequality} 
\[
 f(v, \lambda)=\norm{T(v)- \lambda v} \ge \abs{\lambda} - C\:.
\]
Therefore, if \(\abs{\lambda} > \norm{T(v_0)} + C\) and \(\norm{v}=1\),
\begin{equation}
\label{eqLargeLambda}
  f(v, \lambda) > f(v_0, 0)\:.
\end{equation}
Since \(f\) is continuous and the set 
\[
  \bigl\{(v, \lambda) \in V \times \C \st \norm{v}=1 \text{ and } \abs{\lambda} \le  \norm{T(v_0)} + C\bigr\}
\] 
is compact and not empty, by Weierstrass's theorem there exists \(\Bar{v} \in V\) and \(\Bar{\lambda}\in \C\) with \(\norm{\Bar{v}}=1\) and \(\abs{\Bar{\lambda}} \le  \norm{T(v_0)} + C\) such that for every \(v \in V\) with \(\norm{v}=1\) and \(\lambda \in \C\) with \(\abs{\lambda} \le  \norm{T(v_0)} + C\),
\begin{equation}
\label{eqMinf}
  f(v, \lambda) \ge f(\Bar{v}, \Bar{\lambda})\:.
\end{equation}

By \eqref{eqLargeLambda}, the inequality \eqref{eqMinf} also holds for \(v \in V\) with \(\norm{v}=1\) and \(\lambda \in \C\) with \(\abs{\lambda} > \norm{T(v_0)} + C\).
Finally, if \(v \in V \setminus \{0\}\) and \(\lambda \in \C\), one has 
\[
\frac{\norm{T(v)-\lambda v}}{\norm{v}}= f( \norm{v}^{-1}v, \lambda) \ge f( \Bar{v}, \Bar{\lambda}) = \frac{\norm{T(\Bar{v})-\Bar{\lambda} \Bar{v}}}{\norm{\Bar{v}}}\:. \qedhere
\]
\end{proof}

The second part of the proof is to show that the a minimal pair yields an eigenvector and its associated eigenvalue.

\begin{proposition}
\label{propositionCaractMinimum}
Let \(V\) be a finite-dimensional normed vector space over \(\C\), \(T : V \to V\) be a linear operator, \(\Bar{v} \in V \setminus \{0\}\) and \(\Bar{\lambda} \in \C\). If for every \(v \in V \setminus \{0\}\) and \(\lambda \in \C\),
\[
  \frac{\norm{T(v)-\lambda v}}{\norm{v}}\ge \frac{\norm{T(\Bar{v})-\Bar{\lambda} \Bar{v}}}{\norm{v}}\:,
\]
then \(T(\Bar{v})=\Bar{\lambda} \Bar{v}\).
\end{proposition}

The key point consists in proving that if a minimal pair that does not yield an eigenvector and its associated eigenvalue, then it yields additional minimal pairs.

\begin{lemma}
\label{lemmaMinimumExtension}
Let \(V\) be a finite-dimensional normed complex vector space,  \(T : V \to V\) be a linear operator, \(\Bar{v} \in V \setminus \{0\}\) and \(\Bar{\lambda} \in \C\). If for every \(v \in V \setminus \{0\}\) and \(\lambda \in \C\),
\[
  \frac{\norm{T(v)-\lambda v}}{\norm{v}}\ge \Bar{c}:=\frac{\norm{T(\Bar{v})-\Bar{\lambda} \Bar{v}}}{\norm{v}}\:,
\]
then, for every \(\lambda \in \C\) with \(\abs{\lambda - \Bar{\lambda}} < \Bar{c}\), there exists \(v \in V \setminus \{0\}\) such that 
\[
 \frac{\norm{T(v)-\lambda v}}{\norm{v}}= \Bar{c}\:.
\]
\end{lemma}

The proof relies on the next algebraic computation. Recall that \(I\) denotes the identity map.

\begin{lemma}
\label{lemmaSumInverses}
Let \(V\) be a finite-dimensional normed complex vector space, \(S \colon V \to V\) be a linear operator, \(\sigma \in \C\), \(\omega \in \C\) and \(n \in \N\). If for every \(j \in \{1, \dotsc, n-1\}\), \(\omega^j \ne 1\) and \(\omega^n=1\),
\(S\) is invertible and for every \(j \in \{0, \dotsc, n-1\}\), \( S - \omega^j \sigma I\) is invertible, then
\[
  \sum_{j=0}^{n-1} (S-\omega^j \sigma I)^{-1}\circ \bigl(I- (\sigma S^{-1})^n\bigr)\circ S=n I\:.
\]
\end{lemma}

This formula is proved by L.~Tornheim (in the framework of normed fields) \cite{Tornheim}*{p.~63};	 Kametani states it without proof \cite{Kametani}*{p.~98, last line}.

\begin{proof}[Proof of lemma~\ref{lemmaSumInverses}]
First note that
\begin{equation}
\label{eqISn}
 I- (\sigma S^{-1})^n=(S^n - \sigma^n I)\circ S^{-n}\:. 
\end{equation}
One shows by induction on \(n\) that
\[
  (S-\omega^j \sigma I)^{-1} \circ (S^n- \sigma^n I)=(S-\omega^j \sigma I)^{-1} \circ \bigl(S^n- (\omega ^j\sigma I)^n\bigr)= \sum_{k=0}^{n-1} \omega^{kj}\sigma^{k} S^{n - k - 1}\:.
\]
Hence,
\[
 \sum_{j=0}^{n-1} (S-\omega^j \sigma I)^{-1}\circ (S^{n} - \sigma^n I)
=\sum_{j=0}^{n-1} \sum_{k=0}^{n-1} \omega^{kj}\sigma^{k} S^{n - k - 1}\:.
\]
Now recalling that
\[
 \sum_{j=0}^{n-1}  \omega^{kj} =
  \left\{
   \begin{aligned}
     &\frac{1-\omega^{kn}}{1-\omega^{k\phantom{n}}} = 0 && \text{if \(\omega^k \ne 1\) },\\
     &n && \text{if \(\omega^k = 1\) },
   \end{aligned}
  \right.
\]
we have
\[
 \sum_{j=0}^{n-1} (S-\omega^j \sigma I)^{-1}\circ (S^{n} - \sigma^n I)= n S^{n-1}\:,
\]
from which we conclude with the help of \eqref{eqISn}.
\end{proof}

\begin{proof}[Proof of lemma~\ref{lemmaMinimumExtension}]
We can assume that \(\Bar{c} > 0\). For every \(v \in V \setminus \{0\}\) and \(\lambda \in \C\),
\[
  \norm{T(v)-\lambda v} \ge \Bar{c}\norm{v}\:, 
\]
hence \( T- \lambda I \) is one-to-one and invertible, and  for every \(v \in V \setminus \{0\}\) and \(\lambda \in \C\),
\begin{equation}
\label{ineqinv}
  \norm{(T - \lambda I)^{-1} (v)} \le \frac{\norm{v}}{\Bar{c}}\:, 
\end{equation}
Let \(n \in \N\) and set \(\omega=\cos \frac{2\pi}{n} + i \sin \frac{2\pi}{n}\). By lemma~\ref{lemmaSumInverses} with \(S=T-\lambda I\) and \(\sigma = \Bar{\lambda} - \lambda\) and by the triangle inequality, we have
\[
 n \norm{\Bar{v}}\le \sum_{j=0}^{n-1} \bignorm{\bigl((T-\Bar{\lambda}I-\omega^j (\lambda -\Bar{\lambda}) I)^{-1}\circ \bigl(I- ((\lambda -\Bar{\lambda})(T-\Bar{\lambda}I)^{-1})^n\bigr)\circ (T-\Bar{\lambda}I)\bigr)(\Bar{v})}.
\]
Now, one has for \(j \in \{1, \dotsc, n-1\}\) by \eqref{ineqinv}
\begin{multline*}
\bignorm{\bigl((T-\Bar{\lambda}I-\omega^j (\lambda -\Bar{\lambda}) I)^{-1}\circ \bigl(I- ((\lambda -\Bar{\lambda})(T-\Bar{\lambda}I)^{-1})^n\bigr)\circ (T-\Bar{\lambda}I)\bigr)(\Bar{v})}\\
\le \frac{1}{\Bar{c}} \bignorm{\bigl(I- ((\lambda -\Bar{\lambda})(T-\Bar{\lambda}I)^{-1})^n\bigr)\circ (T-\Bar{\lambda}I)\bigr)(\Bar{v})}\\
\le \frac{1}{\Bar{c}} \Bigl(\bignorm{T(\Bar{v})-\Bar{\lambda}\Bar{v}}+ \bignorm{\bigl((\lambda -\Bar{\lambda})(T-\Bar{\lambda}I)^{-1})^n\bigr)\circ (T-\Bar{\lambda}I)\bigr)(\Bar{v})}\Bigr).
\end{multline*}
On the other hand we have by  \eqref{ineqinv}
\begin{multline*}
\bignorm{\bigl((T-\Bar{\lambda}I- (\lambda -\Bar{\lambda}) I)^{-1}\circ \bigl(I- ((\lambda -\Bar{\lambda})(T-\Bar{\lambda}I)^{-1})^n\bigr)\circ (T-\Bar{\lambda}I)\bigr)(\Bar{v})}\\
\le \bignorm{\bigl((T-\lambda I)^{-1}\circ (T-\Bar{\lambda} I)\bigr)(\Bar{v})}\\+\frac{1}{\Bar{c}} \bignorm{\bigl((\lambda -\Bar{\lambda})(T-\Bar{\lambda}I)^{-1})^n\bigr)\circ (T-\Bar{\lambda}I)\bigr)(\Bar{v})}\:.
\end{multline*}
Therefore, using the previous inequalities, we obtain
\begin{multline*}
n \bignorm{T(\Bar{v})-\Bar{\lambda}\Bar{v}} = n \Bar{c} \norm{\Bar{v}}\\
 \le (n - 1) \bignorm{T(\Bar{v})-\Bar{\lambda}\Bar{v}} + \Bar{c} \bignorm{\bigl((T-\lambda I)^{-1}\circ (T-\Bar{\lambda} I)\bigr)(\Bar{v})}\\
 + n \bignorm{\bigl((\lambda -\Bar{\lambda})(T-\Bar{\lambda}I)^{-1})^n\bigr)\circ (T-\Bar{\lambda}I)\bigr)(\Bar{v})}\:.
\end{multline*}
Applying \(n\) times \eqref{ineqinv} we conclude that
\begin{multline*}
\bignorm{T(\Bar{v})-\Bar{\lambda}\Bar{v}}\\
 \le \Bar{c} \bignorm{\bigl((T-\lambda I)^{-1}\circ (T-\Bar{\lambda} I)\bigr)(\Bar{v})} + n \bignorm{\bigl((\lambda -\Bar{\lambda})(T-\Bar{\lambda}I)^{-1})^n\bigr)\circ (T-\Bar{\lambda}I)\bigr)(\Bar{v})}\\
\le \Bar{c} \bignorm{\bigl((T-\lambda I)^{-1}\circ (T-\Bar{\lambda} I)\bigr)(\Bar{v})}+  n\Bigl(\frac{\abs{\Bar{\lambda}-\lambda}}{\Bar{c}}\Bigr)^n \norm{T(\Bar{v})-\lambda \Bar{v}}\:.
\end{multline*}
Since \(\abs{\lambda - \Bar{\lambda} } < \Bar{c} \), by letting \(n \to \infty\) and taking \(v = \bigl((T-\lambda I)^{-1}\circ (T-\Bar{\lambda} I)\bigr)(\Bar{v})\), we have 
\[
 \norm{T(v)-\lambda v} \le \Bar{c}\bignorm{v}\:,
\]
from which the conclusion follows.
\end{proof}

\begin{proof}[Proof of proposition~\ref{propositionCaractMinimum}]
Assume by contradiction that \(T(\Bar{v})\ne \Bar{\lambda}\Bar{v}\). 
By lemma~\ref{lemmaMinimumExtension}, for every \( \lambda \in \C\) with \(\abs{\lambda - \Bar{\lambda}} < \Bar{c}\), there exists \(v \in V\) such that 
\begin{equation}
\label{eqNewv}
 \Bar{c}=\frac{\norm{T(v)-\lambda v}}{\norm{v}}\:.
\end{equation}
Applying again lemma~\ref{lemmaMinimumExtension} inductively, for every \( \lambda \in \C\) and \(n \in \N\) such that \(\abs{\lambda - \Bar{\lambda}} < n \Bar{c}\), there exists \(v \in V\) such that \eqref{eqNewv} holds.

On the other hand, by lemma~\ref{lemmaInequality}, if \(\abs{\lambda} > \Bar{c}+C\),
\[
  \frac{\norm{T(v)-\lambda v}}{\norm{v}} \ge \abs{\lambda} - C > \Bar{c}\:,
\]
in contradiction with \eqref{eqNewv}.
\end{proof}

\section*{Acknowledgements}

I would like to thank Augusto Ponce for many discussions about teaching analysis and Laure Ninove for many discussion on (determinant-free) linear algebra.

\end{document}